\documentclass{amsart}
\usepackage{amsmath,amsfonts,latexsym,amssymb}

\newtheorem{theorem}{Theorem}
\newtheorem{lemma}{Lemma}

\theoremstyle{remark}
\newtheorem{remark}{Remark}
\newtheorem*{ack}{Acknowledgements}

\newcommand{\Sym}{\text{Sym}^2}

\begin{document}

\markboth{Ritabrata Munshi}{Bounds for twisted symmetric square $L$-functions - III}
\title[Bounds for twisted symmetric square $L$-functions - III]{Bounds for twisted symmetric square $L$-functions - III}

\author{Ritabrata Munshi}   
\address{School of Mathematics, Tata Institute of Fundamental Research, 1 Homi Bhabha Road, Colaba, Mumbai 400005, India.}     
\email{rmunshi@math.tifr.res.in}

\begin{abstract}
Let $f$ be a newform, and let $\chi$ be a primitive character of conductor $q^{\ell}$. Assume that $q$ is an odd prime. In this paper we prove the subconvex bound
$$
L\left(\tfrac{1}{2},\Sym f\otimes\chi\right)\ll_{f,q,\varepsilon} q^{3\ell\left(\frac{1}{4}-\frac{1}{36}+\varepsilon\right)}
$$
for any $\varepsilon>0$. This can be compared with the recently established $t$-aspect subconvexity of the symmetric square $L$-functions.
\end{abstract}


\maketitle


\section{Introduction}
\label{intro}

Let $f$ be a newform of full level and weight $k$. Let $\chi$ be a primitive character of conductor $M$. Then from the functional equation of $L(s,\Sym f\otimes\chi)$ and the convexity principle we get the bound
$$
L\left(\tfrac{1}{2},\Sym f\otimes\chi\right)\ll_{f,\varepsilon} M^{\frac{3}{4}+\varepsilon}
$$
for the central value. Getting a subconvex bound of the form $M^{\theta}$ with $\theta<3/4$, in this context is an intriguing open problem, which has recently attracted some attention. Blomer \cite{Bl2} has proved such a subconvex bound for quadratic characters $\chi$ with prime modulus. In \cite{Mu} we considered the case of general (non-quadratic) characters with prime power modulus $q^{\ell}$ for any fixed $\ell>1$. In particular for $\chi$ primitive modulo $q^3$ ($q$ prime) it was shown that
$$
L\left(\tfrac{1}{2},\Sym f\otimes\chi\right)\ll_{k,\varepsilon} q^{\frac{9}{4}-\frac{1}{4}+\varepsilon}.
$$
In the sequel \cite{Mu-2} we used similar methods to prove subconvex bounds when the modulus of $\chi$ is square-free with a suitably large prime factor. The purpose of the present paper, the third in the series, is to prove the following
\begin{theorem}
\label{mthm}
Let $f\in S_{k}(1)$ be a newform of full level and weight $k$. Let $\chi$ be a character of conductor $q^{\ell}$ where $q$ is an odd prime number. Then we have
$$
L\left(\tfrac{1}{2},\text{$\rm{Sym}^2$} f\otimes\chi\right)\ll_{f,q,\varepsilon} q^{\frac{3}{4}\ell-\frac{1}{12}\ell+\ell\varepsilon}.
$$
The implied constant depends on $f$, $q$ and $\varepsilon$, but does not depend on $\ell$. 
\end{theorem}

\begin{remark}
In \cite{Mu-2} we considered the case where $\ell$ is fixed and $q\rightarrow \infty$, and here we are looking at the case where $q$ is fixed and $\ell\rightarrow\infty$.  
\end{remark}

\begin{remark}
From the adelic point of view, this result can be compared with the $t$-aspect subconvexity result of Li \cite{L} for the symmetric square $L$-function. Any (unitary) Hecke character  $\Psi$ on the idele group 
$$
\mathbb A_{\mathbb Q}^\times/\mathbb Q^\times=\mathbb R_+\times\prod_p\mathbb Z_p^\times
$$ 
decomposes as $\Psi=|.|^{it}\otimes\left(\otimes_p\psi_p\right)$ where $t\in \mathbb R$, and $\psi_{\textrm{f}}=\otimes_p\psi_p$ corresponds to a Dirichlet character, $\chi$ say. Then the twisted $L$-function $L\left(\tfrac{1}{2},\text{$\rm{Sym}^2$} f\otimes\Psi\right)$ corresponds to $L\left(\tfrac{1}{2}+it,\text{$\rm{Sym}^2$} f\otimes\chi\right)$. Thus Li \cite{L} deals with the case of twists by Hecke characters which are `supported' only at the prime at infinity, and in the present case we are considering twists by Hecke characters which are `supported' (ramified) at a fixed finite prime $q$.
\end{remark}

\ack
A part of this work was done while the author was enjoying the hospitality of the Indian Statistical Institute, Kolkata. The author also wishes to thank Peter Sarnak for suggesting him to look at this aspect of subconvexity. 


\section{Preliminaries}
\label{prelim}

In this section we will briefly recall some fundamental facts about holomorphic forms and their $L$-functions (for details see \cite{IK}). Let $f\in S_k(1)$ be a newform with Fourier expansion 
$$
f(z)=\sum_{n=1}^{\infty}\lambda_f(n)n^{\frac{k-1}{2}}e(nz).
$$ 
For $s=\sigma+it$ with $\sigma>1$, the associated $L$-function is given by 
$$
L(s,f)=\sum_{n=1}^{\infty}\lambda_f(n)n^{-s}=\prod_p \left(1-\alpha_f(p)p^{-s}\right)^{-1} \left(1-\beta_f(p)p^{-s}\right)^{-1}.
$$
The local parameter $\alpha_f(p)$ and $\beta_f(p)$ are related to the normalized Fourier coefficients by $\alpha_f(p)+\beta_f(p)=\lambda_f(p)$ and $\alpha_f(p)\beta_f(p)=1$. Now let $\chi$ be a primitive Dirichlet character of modulus $M$. Then we define the twisted symmetric square $L$-function by the degree three Euler product 
$$
L(s,\Sym f\otimes\chi)=\prod_p \left(1-\alpha_f^2(p)\chi(p)p^{-s}\right)^{-1} \left(1-\chi(p)p^{-s}\right)^{-1} \left(1-\beta_f^2(p)\chi(p)p^{-s}\right)^{-1},
$$ 
for $\sigma>1$. In this half-plane we have
$$
L(s,\Sym f\otimes\chi)=L(2s,\chi^2)\sum_{n=1}^{\infty}\lambda_f(n^2)\chi(n)n^{-s}.
$$ \\

It is well-known that this $L$-function extends to an entire function and satisfies a functional equation (see \cite{Li}). Indeed we have a completed $L$-function defined as
$$
\Lambda(s,\Sym f\otimes\chi)=M^{3s/2}\gamma(s)L(s,\Sym f\otimes\chi)
$$
where $\gamma(s)$ is essentially a product of three gamma functions $\Gamma(\frac{s+\kappa_j}{2})$, $j=1,2,3$, with $\kappa_j$ depending on the weight of $f$ and the parity of the character $\chi$, such that 
$$
\Lambda(s,\Sym f\otimes\chi)=\varepsilon(f,\chi)\Lambda(1-s,\Sym f\otimes\bar\chi).
$$
Here $\text{Re}(\kappa_j)>0$ and the $\varepsilon$-factor satisfies $|\varepsilon(f,\chi)|= 1$. Using standard arguments we get that the twisted $L$-value $L(\frac{1}{2},\Sym f\otimes\chi)$, is given by a rapidly converging series.
\begin{lemma}
\label{afeqn}
We have
\begin{align}
\label{afe}
L(\tfrac{1}{2},\mathrm{Sym}^2 f\otimes\chi)=\sum_{n=1}^{\infty}\frac{\lambda_f(n^2)\chi(n)}{\sqrt{n}}V\left(\frac{n}{M^{3/2}}\right)+
\varepsilon(f,\chi)\sum_{n=1}^{\infty}\frac{\lambda_f(n^2)\overline{\chi(n)}}{\sqrt{n}}
V\left(\frac{n}{M^{3/2}}\right),
\end{align}
where
$$
V(y)=\frac{1}{2\pi i}\int_{(3)}
\frac{\gamma\left(\frac{1}{2}+u\right)}{\gamma\left(\frac{1}{2}\right)}
\left(\cos \frac{\pi u}{4A}\right)^{-12A}L(1+2u,\chi^2)y^{-u}\frac{du}{u}.
$$ 
\end{lemma}

Here $A$ is a sufficiently large positive integer. The weight function $V(y)$ satisfies the bound 
$$
y^uV^{(u)}(y)\ll_{u,A} y^{-A}.
$$
Breaking the sum in \eqref{afe} into dyadic blocks, it follows that 
$$
L(\tfrac{1}{2},\Sym f\otimes\chi)\ll_{A,\varepsilon} M^{\varepsilon}\sum_{N}\frac{\left|L_f(N)\right|}{\sqrt{N}}\left(1+\frac{N}{M^{3/2}}\right)^{-A}
$$
where $N$ ranges over the values $2^{\alpha}$ with $-1/2\leq \alpha$, and 
$$
L_f(N)=\sum_{n}\lambda_f(n^2)\chi(n)h(n/N).
$$
Here $h(.)$ is a smooth function supported in $[1,2]$. For any $\varepsilon>0$ we can choose $A$ appropriately so that the contribution from $N>M^{3/2+\varepsilon}$ is negligible. Also for the smaller values of $N$ we can estimate the sum trivially. It turns out that the worst case scenario corresponds to the case where $N\asymp M^{3/2}$. \\

Now $f$ can be considered as a Hecke form in the larger space $S_k(L)$ of cusp forms of level $L$. Then we select an orthogonal basis $\mathcal B_k(L)$ of the space $S_k(L)$ containing the form $f$. For any form $g\in S_k(L)$ we have the Fourier expansion $g(z)=\sum \lambda_g(n)n^{\frac{k-1}{2}}e(nz)$. Let 
$$
\left<g_1,g_2\right>_L=\int_{\Gamma_0(L)\backslash\mathbb H}g_1(z)\overline{g_2(z)}y^{k-2}dxdy
$$ 
denote the Petersson inner product at level $L$. Let $\|g\|_L^2=\left<g,g\right>_L$ denote the Petersson norm at level~$L$. 
\begin{lemma}
(Petersson formula) We have 
\begin{align}
\label{pform}
\frac{\Gamma(k-1)}{(4\pi)^{k-1}}\sum_{g\in \mathcal B_k(L)}\frac{\lambda_g(n)\lambda_g(m)}{\|g\|_L^2}=\delta(n,m)+
2\pi i^{-k}\sum_{c=1}^{\infty}\frac{S(n,m;cL)}{cL}J_{k-1}\left(\frac{4\pi \sqrt{nm}}{cL}\right),
\end{align}
where 
$$
S(n,m;cL)=\sideset{}{^\star}\sum_{a\bmod{cL}}e\left(\frac{an+\bar am}{cL}\right)
$$ denotes the Kloosterman sum and $J_{k-1}(.)$ is the $J$-Bessel function of order $k-1$. 
\end{lemma}
It is well known that the Bessel function can be expressed as
\begin{align}
\label{bessel-split}
J_{k-1}(2\pi x)=e(x){W}_k(x)+e(-x)\bar W_k(x)
\end{align}
where $W_k:(0,\infty)\rightarrow \mathbb C$ is a smooth function satisfying the bound
\begin{align}
\label{bessel-bd0}
x^jW_k^{(j)}(x)\ll_j \min\{x^{k-1},x^{-\frac{1}{2}}\}. 
\end{align}


\section{Reciprocity and Poisson summation - I}
\label{rec-poi1}

Let $f\in S_k(1)$ be a Hecke form and let $\chi$ be a character of conductor $q^\ell$, as in the statement of Theorem~\ref{mthm}. Set 
$$
j=2[\theta\ell]\;\;\;\; \text{for some}\;\;\; \theta\in (0,1),
$$ 
which will be specified later. In particular we have $2\ell-j\geq 0$. Let $\mathcal B=\mathcal B_k(16q^{j})$ be an orthogonal basis of $S_k(16q^{j})$ containing the given form $f$. Let $N\leq q^{\frac{3}{2}\ell+\ell\varepsilon}$ and set
$$
S:=\frac{\Gamma(k-1)}{(4\pi)^{k-1}}\sum_{g\in \mathcal B}\frac{1}{\|g\|_{16q^j}^2}\left|\sum_{n\in\mathbb Z}\lambda_g(n^2)\chi(n)h\left(\frac{n}{N}\right)\right|^2.
$$ 
Opening the absolute square and interchanging the order of summation we arrive at
$$
S=\mathop{\sum\sum}_{n,m\in\mathbb Z}\chi(n)\overline{\chi(m)}h\left(\frac{n}{N}\right)h\left(\frac{m}{N}\right)\left[\frac{\Gamma(k-1)}{(4\pi)^{k-1}}\sum_{g\in \mathcal B}\frac{\lambda_g(n^2)\lambda_g(m^2)}{\|g\|_{16q^j}^2}\right].
$$ 
Now to the innermost sum we apply the Petersson formula \eqref{pform}. The contribution from the diagonal is dominated by
$\sum_nh\left(\frac{n}{N}\right)^2\ll N$.\\ 

Now we turn our attention to the off-diagonal which is given by
\begin{align}
\label{offdiag}
S_O=\sum_{\substack{c=1\\16|c}}^{\infty}\frac{1}{q^jc}\mathop{\sum\sum}_{n,m\in\mathbb Z}\chi(n)\overline{\chi(m)}h\left(\frac{n}{N}\right)h\left(\frac{m}{N}\right)S(n^2,m^2;q^jc)J_{k-1}\left(\frac{4\pi nm}{q^j c}\right).
\end{align}
Using a smooth partition of unity we break the sum over $c$ into dyadic blocks and analyse the contribution of each blocks
\begin{align}
\label{offdiag-dy}
S_O(C)=\sum_{\substack{c=1\\16|c}}^{\infty}\frac{1}{q^jc}\mathop{\sum\sum}_{n,m\in\mathbb Z}\chi(n)\overline{\chi(m)}h\left(\frac{n}{N}\right)h\left(\frac{m}{N}\right)S(n^2,m^2;q^jc)J_{k-1}\left(\frac{4\pi nm}{q^j c}\right)G\left(\frac{c}{C}\right).
\end{align}
Here $G(x)$ is a smooth function supported in $[1,2]$. The transition range for the Bessel function is marked by $C\sim N^2/q^j$, and  we define $B$ by setting $C=\frac{N^2}{q^jB}$. Of course the real challenge lies in dealing with the case where $C$ is near the transition range. However for smaller values of $C$ there are complications arising from the oscillation of the Bessel function. \\

We write $c=q^rc'$ where $(c',q)=1$. Then the Kloosterman sum splits as
$$
S(n^2,m^2;q^jc)=S(\overline{c'}n^2,\overline{c'}m^2;q^{j+r})S(\overline{q^{j+r}}n^2,\overline{q^{j+r}}m^2;c').
$$
Observe that in \eqref{offdiag-dy} $nm$ is coprime with $q$ due to the presence of the character $\chi$. So it follows that
$$
S(\overline{c'}n^2,\overline{c'}m^2;q^{j+r})=S(\overline{c'}nm,\overline{c'}nm;q^{j+r})=2\left(\frac{c'nm}{q}\right)^{r}q^{\frac{j+r}{2}}\:\text{Re} \:\varepsilon_{q^{j+r}}e\left(\frac{2\overline{c'}nm}{q^{j+r}}\right)
$$
where $\varepsilon_{q^{j+r}}$ is the sign of the quadratic Gauss sum modulo $q^{j+r}$. With this $S_O(C)$ splits into a sum of two similar terms, each of which is an infinite sum parameterized by $r$. A representative term in this sum is given by
\begin{align}
\label{sum1}
\frac{1}{q^{\frac{j+r}{2}}}\sum_{\substack{(c,q)=1\\16|c}}&\frac{1}{c}G\left(\frac{q^rc}{C}\right)\mathop{\sum\sum}_{n,m\in\mathbb Z}\chi(n)\overline{\chi(m)}\left(\frac{cnm}{q}\right)^re\left(\frac{2\overline{c}nm}{q^{j+r}}\right)S(\overline{q^{j+r}}n^2,\overline{q^{j+r}}m^2;c)\\
\nonumber &\times h\left(\frac{n}{N}\right)h\left(\frac{m}{N}\right)J_{k-1}\left(\frac{4\pi nm}{q^{j+r} c}\right).
\end{align}\\

We start by considering the sums over $n$ and $m$. Applying reciprocity
$$
e\left(\frac{2\overline{c}nm}{q^{j+r}}\right)=e\left(\frac{-2\overline{q^{j+r}}nm}{c}\right)
e\left(\frac{2nm}{q^{j+r}c}\right)
$$
we reduce the modulus. We club the second factor with the Bessel function in \eqref{sum1}. Using the expression \eqref{bessel-split} we see that 
$$
e\left(\frac{2nm}{q^{j+r}c}\right)J_{k-1}\left(\frac{4\pi nm}{q^{j+r}c}\right)=e\left(\frac{4nm}{q^{j+r}c}\right)W_k\left(\frac{2nm}{q^{j+r}c}\right)+\bar W_k\left(\frac{2nm}{q^{j+r}c}\right).
$$
The second factor on the right hand side is without oscillation, and hence is more tamed compared to the first factor. We shall now continue our analysis with the first factor. The analysis with the second factor, which we are omitting, is much simpler. At the end it turns out that the bound that we obtain for the contribution of the second factor is better than that of the first factor. \\

We define 
$
\Phi(x,y;c)=h\left(x\right)h\left(y\right)e\left(\frac{4xy}{c}\right)W_{k}\left(\frac{2xy}{c}\right)
$
and set 
\begin{align}
\label{tr}
\mathcal T_{r,C}=\frac{1}{q^{\frac{j+r}{2}}}\sum_{\substack{c=1\\q\nmid c}}^{\infty}\frac{T_r(16c)}{16c}G\left(\frac{q^rc}{C}\right),
\end{align}
where $T_r(c)$ is defined by
\begin{align*}
\mathop{\sum\sum}_{n,m\in \mathbb Z}\chi(n)\overline{\chi(m)}\left(\frac{cnm}{q}\right)^re\left(\frac{-2\overline{q^{j+r}}nm}{c}\right)S(\overline{q^{j+r}}n^2,\overline{q^{j+r}}m^2,c)\Phi\left(\frac{n}{N},\frac{m}{N};\frac{q^{j+r}c}{N^2}\right).
\end{align*}
Now we are ready to apply the Poisson summation formula on both the sums over $n$ and $m$. \\

\begin{lemma}
\label{poisson}
We have
\begin{align}
\label{trrt}
T_r(c)=\frac{N^2}{q^{2\ell}c^2}\mathop{\sum\sum}_{n,m\in\mathbb Z}D_r(n,m;c)I_r(n,m;c),
\end{align}
where the character sum $D_r(n,m;c)$ is given by
$$
\mathop{\sum\sum}_{\alpha,\beta \bmod q^\ell c}\chi(\alpha)\overline{\chi(\beta)}\left(\frac{\alpha\beta}{q}\right)^re\left(\frac{-2\overline{q^{j+r}}\alpha\beta}{c}\right)S(\overline{q^{j+r}}\alpha^2,\overline{q^{j+r}}\beta^2,c)e\left(\frac{\alpha n+ \beta m}{q^\ell c}\right),
$$
and the integral $I_r(n,m;c)$ is given by
$$
\mathop{\iint}_{\mathbb R^2}h\left(x\right)h\left(y\right)e\left(\frac{4xyN^2}{q^{j+r}c}\right)W_k\left(\frac{2xyN^2}{q^{j+r}c}\right)e\left(\frac{-(nx+my)N}{q^\ell c}\right)dxdy.
$$
\end{lemma}
\begin{proof}
First we break the sum over $n$ and $m$ into congruence classes modulo $q^\ell c$ to get
\begin{align*}
\mathop{\sum\sum}_{\alpha,\beta \bmod q^\ell c}\chi(\alpha)\overline{\chi(\beta)}&\left(\frac{\alpha\beta}{q}\right)^re\left(\frac{-2\overline{q^{j+r}}\alpha\beta}{c}\right)S(\overline{q^{j+r}}\alpha^2,\overline{q^{j+r}}\beta^2,c)\\
&\times \mathop{\sum\sum}_{n,m\in\mathbb Z}\Phi\left(\frac{\alpha+nq^\ell c}{N},\frac{\beta+mq^\ell c}{N};\frac{q^{j+r}c}{N^2}\right).
\end{align*}
Now applying Poisson summation formula and making a change of variables, we get that the inner sum over $(n,m)$ is given by
$$
\frac{N^2}{q^{2\ell}c^2}\mathop{\sum\sum}_{n,m\in\mathbb Z}e\left(\frac{\alpha n+ \beta m}{q^\ell c}\right)
\mathop{\iint}_{\mathbb R^2}\Phi\left(x,y;\frac{q^{j+r}c}{N^2}\right)
e\left(\frac{-(nx+my)N}{q^\ell c}\right)dxdy.
$$
The lemma follows after rearranging the order of summations and integration.
\end{proof}

We conclude this section with an observation. Using integration-by-parts on $I_r(n,m;c)$ and the analytic properties of the Bessel function we get that we can basically ignore the contribution coming from the terms in \eqref{tr} with $c$ large, e.g. $c>Q=q^{2012\ell}$, or from the terms in \eqref{trrt} with $\text{gcd}(n,m)>Q$.


\section{The character sum $D_r(n,m;c)$ and the integral $I_r\left(n,m;c\right)$}
\label{sec-char}

In this section we will first explicitly compute the character sum $D_r(n,m;c)$ which appears in Lemma~\ref{poisson}. Since $(c,q)=1$, we have the factorization
\begin{align}
\label{char-sum-d-r}
D_r(n,m;c)=A_r(\bar cn,\bar cm;q^\ell)B_r(\overline{q^\ell}n,\overline{q^\ell}m;c)
\end{align}
where
$$
A_r(n,m;q^\ell)=\mathop{\sum\sum}_{\alpha,\beta \bmod q^\ell}\chi(\alpha)\overline{\chi(\beta)}\left(\frac{\alpha\beta}{q}\right)^re\left(\frac{\alpha n+ \beta m}{q^\ell}\right),
$$
and
\begin{align}
\label{char-b}
B_r(n,m;c)=\mathop{\sum\sum}_{\alpha,\beta \bmod c}e\left(\frac{-2\overline{q^{j+r}}\alpha\beta}{c}\right)S(\overline{q^{j+r}}\alpha^2,\overline{q^{j+r}}\beta^2,c)e\left(\frac{\alpha n+\beta m}{c}\right).
\end{align}
The first character sum $A_r(\bar cn,\bar cm;q^\ell)$, in fact, does not depend on $c$. 
\begin{lemma}
\label{char1}
We have
$$
A_r(\bar cn,\bar cm;q^\ell)=\overline{\chi(n)}\chi(-m)\left(\frac{-nm}{q}\right)^r\left|\sum_{\alpha\bmod q^\ell}\chi(\alpha)\left(\frac{\alpha}{q}\right)^re\left(\frac{\alpha}{q^\ell}\right)\right|^2.
$$\\
\end{lemma}

To evaluate the other character sum we follow the exposition in Section 4 of \cite{Mu}. Let $c=d2^{\eta}$, with $d$ odd, we define 
$$
C_{\pm}(n,m;c)=\sideset{}{^\star}\sum_{\substack{a\bmod c\\an\equiv m \bmod c}}\left(\frac{a}{d}\right)\psi_{\pm}(a),
$$
where 
$$
\psi_{+}=\begin{cases}\chi_0 &\text{if $\eta$ is even,}\\\chi_{8} &\text{if $\eta$ is odd,}\end{cases}
$$ 
and $\psi_{-}=\psi_{+}\chi_{-4}$. Let $(n,m)=\delta$, and write $n=\delta n^*$ and $m=\delta m^*$. Also write $c=c_1c_2$ with $c_1|(nm)^{\infty}$, $(c_2,nm)=1$ and $c_1=c_{11}c_{12}^2$ with $c_{11}$ square-free. (Recall that the notation $a|b^{\infty}$ means that $a$ divides a sufficiently large power of $b$.) It follows that the sum vanishes unless $c_1|\delta^{\infty}$. Moreover $C_\pm(n,m;c)$ is multiplicative as function in $c$, and for $c=p^u$ the character sum can be easily computed. It follows that
$$
C_{\pm}(n,m;c)=\left(\frac{n^*m^*}{c_{11}c_2}\right)\psi_{\pm}(n^*m^*)\tilde C(\delta;c_1),
$$
where $\tilde C(\delta;c_1)$ depends only on $\delta$, the gcd of $n$ and $m$, and on $c_1$. Moreover we have the bound
$$
\left|\tilde C(\delta;c_1)\right|\leq \prod_p \mathcal M(p^{v_p(\delta)},p^{v_p(c_1)}),
$$
where $\mathcal M(p^u,p^v)=0$ if $v$ is odd and $u\geq v$, $\mathcal M(p^u,p^v)=p^u$ if $v$ is odd and $u<v$, and $\mathcal M(p^u,p^v)=\min\{p^u,p^v\}$ if $v$ is even. The next lemma follows using Hecke's trick.
\begin{lemma}
\label{heck-trick}
We have
$$
\sum_{\delta\leq Q}\frac{1}{\delta} \sum_{\substack{c_1\leq Q\\c_1|\delta^\infty}}\frac{|\tilde C(\delta;c_1)|}{c_{12}}\ll Q^{\varepsilon}.
$$
\end{lemma}
Now we are ready to evaluate the character sum $B_r(n,m;c)$. Set $\varepsilon_d = 1$ if $d\equiv 1 \bmod 4$, and $\varepsilon_d = i$ if $d\equiv 3 \bmod 4$. 
\begin{lemma}
\label{char2}
Let $c=2^{\eta}d$ with $d$ odd, $(q,d)=1$ and $\eta\geq 4$. Then $B_r(n,m;c)=0$ unless $4|(n,m)$, and $(m,c)=(n,c)$. Write $n=2n'$ and $m=2m'$. Then we have 
$$
B(\overline{q^\ell}n,\overline{q^\ell}m;c)=c^{\frac{3}{2}}\varepsilon_d\:e\left(\frac{\overline{q^{2\ell}}q^{j+r}n'm'}{c}\right)\left\{C_{+}(q^{j+r}n,-m;c)+i\chi_{-4}(d)C_{-}(q^{j+r}n,-m;c)\right\}.
$$\\
\end{lemma}

\vspace{.5cm}

Now we will study the integral 
\begin{align}
\label{int-again}
I_r\left(n,m;c\right)=\mathop{\iint}_{\mathbb R^2}h\left(x\right)h\left(y\right)e\left(\frac{4xyN^2}{q^{j+r}c}\right)W_k\left(\frac{2xyN^2}{q^{j+r}c}\right)e\left(\frac{-(nx+my)N}{q^\ell c}\right)dxdy.
\end{align}
which appears in Lemma \ref{poisson}. We will now obtain bounds for this integral using repeated integration by parts. Recall that $cq^r\sim C$ and $C=\frac{N^2}{q^jB}$. So we will write $c=Cz/q^r$ for some $z\in [1,2]$. Differentiating the first four factors and integrating the last factor we get
$$
I_r\left(n,m;\frac{Cz}{q^r}\right)\ll_{i_1,i_2} \left(1+\frac{N^2}{q^jC}\right)^{i_1+i_2}\left(\frac{q^\ell C}{|n|Nq^r}\right)^{i_1}\left(\frac{q^\ell C}{|m|Nq^r}\right)^{i_2} \;\;\;\;\text{for}\;\;i_1, i_2\geq 0.
$$
So it follows that the above integral is negligibly small unless 
\begin{align}
\label{nm-bound}
|n|, |m| \ll \frac{N}{q^{j+r-\ell}}\left(B^{-1}+1\right)q^{\ell\varepsilon}.
\end{align}\\

For $B>q^{\ell\varepsilon}$ there is oscillation in the third factor on the right hand side of \eqref{int-again}. In this case we may also do repeated integration by parts by integrating the third factor and differentiating the other factors. This process yields the bound
$$
I_{r}\left(n,m;\frac{Cz}{q^r}\right)\ll_{i_1,i_2} \left(1+\frac{|n|Nq^r}{q^\ell C}\right)^{i_1}\left(1+\frac{|m|Nq^r}{q^\ell C}\right)^{i_2}\left(\frac{q^jC}{N^2}\right)^{i_1+i_2}\;\;\;\;\text{for}\;\;i_1, i_2\geq 0.
$$
Since we are assuming that $C=\frac{N^2}{q^jB}< \frac{N^2}{q^{j+\ell\varepsilon}}$, it follows that the integral is arbitrarily small unless $|n|, |m| \gg \frac{N}{q^{j+r-\ell}}q^{-\ell\varepsilon}$. 
\begin{lemma}
\label{int-000}
Suppose $C<\frac{N^2}{q^{j+\ell\varepsilon}}$ (or in other words $B>q^{\ell\varepsilon}$) then the integral $I_{r}\left(n,m;\tfrac{Cz}{q^r}\right)$ is negligibly small unless 
$$
|n|, |m| \in \left[\frac{N}{q^{j+r-\ell}}q^{-\ell\varepsilon},\frac{N}{q^{j+r-\ell}}q^{\ell\varepsilon}\right].
$$\\
\end{lemma}

We end this section by noting the following consequence of the above bounds. 
\begin{lemma}
\label{mid-lemma}
For $r>2\ell-j-2$ and any $C$, we have
$$
\mathcal T_{r,C}\ll Nq^{\ell\varepsilon}.
$$
\end{lemma}
Consequently from the next section onwards we shall only consider the case $r\leq 2\ell-j-2$. 


\section{Reciprocity and Poisson summation - II}
\label{sec-poi}

In Section \ref{sec-char}, we explicitly computed the character sum which appears in Lemma \ref{poisson}. Substituting this explicit form of the character sum in \eqref{tr}, and ignoring the negligible contribution that comes from the large values of $c$ and the large values of the gcd $(n,m)$ (as we have noted after Lemma \ref{poisson}), we are basically left with the job of analysing sums of the type 
\begin{align}
\label{tr'}
\mathcal T_{r,C}^\star=\frac{N^2}{q^{\ell+\frac{j+r}{2}}}\sum_{\substack{\delta\leq Q\\q\nmid \delta}}
\sum_{\substack{c_1\leq Q\\c_1|(2\delta)^{\infty}}}\left|\tilde C(\delta; c_1)\right|\mathop{\sum\sum}_{\substack{1\leq u,v\leq Q\\u,v|(2\delta)^{\infty}\\(u,v)=1}}\Bigl|T_{r,C}^\star(\delta,c_1,u,v)\Bigr|
\end{align}
where $c_1=c_{11}c_{12}^2$ with $c_{11}$ square-free,
\begin{align}
\label{tr''}
T_{r,C}^\star(\delta,c_1,u,v)=\mathop{\sum\sum}_{\substack{n,m=1\\(n,m)=1\\(nm,\delta q)=1\\ n,m\equiv 1 \bmod 4}}^{\infty}
\overline{\chi(n)}\chi(m)\left(\frac{nm}{c_{11}q^r}\right)S_{r,c_1}^\star(\delta un,\delta vm),
\end{align}
and 
\begin{align}
\label{S_r}
S_{r,c_1}^\star(n,m)=\sum_{\substack{c_2=1\\(c_2,2qnm)=1}}^{\infty}\left(\frac{nmq^r}{c_2}\right)e\left(\frac{\overline{q^{2\ell}}q^{j+r}nm}{c_1c_2}\right)\frac{I_{r}(2n,2m;c_1c_2)}{(c_1c_2)^{3/2}}G\left(\frac{c_1c_2q^r}{C}\right).
\end{align}
(Here $Q=q^{2012\ell}$.) Recall that by our choice $2\ell-j\geq 0$. \\

The assumption on $n$ and $m$, that they are $\equiv 1 \bmod 4$, is made to simplify some of the standard complications related to the prime $2$. In general we may take out the $2$-primary part from $n$ or $m$, and then split the sum into four parts depending on the possible congruence classes modulo $4$. Then for each sum we follow the same steps that we take below. Also note that for notational simplicity we are just focusing on the contribution from the nonnegative $n$ and $m$. \\

Our next step is an application of the Poisson summation formula on the sum over $c_2$ in \eqref{S_r}. To this end we have to first apply reciprocity
$$
e\left(\frac{\overline{q^{2\ell-j}}q^{r}nm}{c_1c_2}\right)=
e\left(\frac{-\overline{c_1c_2}q^{r}nm}{q^{2\ell-j}}\right)
e\left(\frac{q^{r}nm}{q^{2\ell-j}c_1c_2}\right).
$$
We will include the last factor in our smooth function. Set
$$
\mathcal I_{r}(n,m;c)=e\left(\frac{q^{r}nm}{q^{2\ell-j}c}\right)\frac{I_{r}(2n,2m;c)}{c^{3/2}}G\left(\frac{cq^r}{C}\right).
$$ 
Now consider the sum 
\begin{align}
\label{S_r*}
S_{r,c_1}^\star(n,m)=\sum_{\substack{c_2\in\mathbb Z\\(c_2,2qnm)=1}}\left(\frac{c_2}{nmq^r}\right)e\left(\frac{-\overline{c_1c_2}q^{r}nm}{q^{2\ell-j}}\right)\mathcal I_{r}(n,m;c_1c_2).
\end{align}
Here we have applied quadratic reciprocity. This is one of the places where we use the assumption that $n,m\equiv 1 \bmod 4$. We will use this yet another time in the evaluation of the character sum that appears in our next result.
\begin{lemma}
\label{poisson2}
We have
$$
S_{r,c_1}^\star(n,m)=\frac{q^{\frac{r}{2}}}{2c_1q^{2\ell-j}\sqrt{C}[n,m]}\sum_{c_2\in\mathbb Z}E_{r,c_1}(c_2;n,m)\tilde{\mathcal I}_{r,c_1}(c_2;n,m)
$$
where the character sum is given by 
$$
E_{r,c_1}(c_2;n,m)=\sideset{}{^\star}\sum_{\beta \bmod 2q^{2\ell-j}[n,m]}\left(\frac{\beta }{nmq^r}\right)e\left(\frac{-\overline{c_1\beta}q^{r}nm}{q^{2\ell-j}}\right)e\left(\frac{\beta c_2}{2q^{2\ell-j}[n,m]}\right),
$$
and
\begin{align*}
\tilde{\mathcal I}_{r,c_1}(c_2;n,m)=\int G(z)I_{r}\left(2n,2m;\frac{Cz}{q^r}\right)
e\left(\frac{q^{2r}nm}{q^{2\ell-j}Cz}-\frac{c_2Cz}{2q^{2\ell-j+r}c_1[n,m]}\right)\frac{dz}{z^{\frac{3}{2}}}.
\end{align*}
\end{lemma}
\begin{proof}
This is a standard application of the Poisson summation formula.
\end{proof}


\section{The character sum $E_{r,c_1}(c_2;n,m)$ and the integral $\tilde{\mathcal I}_{r,c_1}(c_2;n,m)$}
\label{sec-char2}

Next we will explicitly evaluate the character sum which appears in Lemma \ref{poisson2}. For \eqref{tr'}, we only require to consider the character sum $E_{r,c_1}(c_2;u\delta n,v\delta m)$ where $uv|(2\delta)^{\infty}$, $(u,v)=(n,m)=1$, $n,m\equiv 1\bmod{4}$ and $(nm,\delta q)=1$. In this case the sum splits into a product given by 
\begin{align}
\label{split-sum}
\left(\frac{\delta}{nm}\right)\left[\:\sideset{}{^\star}\sum_{\beta \bmod 2\delta uv}\left(\frac{\beta }{uv}\right)e\left(\frac{\beta c_2}{2\delta uv}\right)\right]&\left[\:\sideset{}{^\star}\sum_{\beta \bmod nm}\left(\frac{\beta}{nm}\right)e\left(\frac{\beta c_2}{nm}\right)\right]\\
\nonumber &\times\left[\:\sideset{}{^\star}\sum_{\beta \bmod q^{2\ell-j}}\left(\frac{\beta }{q^r}\right)e\left(\frac{-\overline{c_1\beta}q^{r}\delta w+\overline{2w}\beta c_2}{q^{2\ell-j}}\right)\right],
\end{align}
where $w=[\delta un,\delta vm]=\delta uvnm$. The first two sums are just Gauss sums and the last sum is a generalized Kloosterman sum. We denote the first sum by $g_{\delta,u,v}^\star(c_2)$, and note the following bound:
\begin{lemma}
\label{bound-for-g}
Let $\delta=\delta_1\delta_2^2$ and $c_2=c_{21}c_{22}^2$ with $\delta_1$, $c_{21}$ square-free. Then we have 
$$
g_{\delta,u,v}^\star(c_2)\ll uv\delta_2(\delta_1,c_{21})(\delta_1\delta_2,c_{22}).
$$
\end{lemma}
\begin{proof}
First using multiplicativity we reduce to the case where $\delta$, $u$ and $v$ are powers of a given prime $p$. Then we use well-known bounds for the Gauss sums and Ramanujan sums. Finally we need to verify that the power of the prime $p$ which appears on the right hand side is sufficiently large. For this we need to consider several cases. We prefer to omit the details. 
\end{proof}

The middle sum in \eqref{split-sum} is given by $g(n,c_2)g(m,c_2)$ where $g(n,c)$ stands for the usual Gauss sum 
$$
g(n,c)=\sideset{}{^\star}\sum_{\beta \bmod n}\left(\frac{\beta }{n}\right)e\left(\frac{\beta c}{n}\right).
$$
The other character sum modulo $q^{2\ell-j}$, after a change of variable is given by 
$$
\left(\frac{\delta uvnm}{q^r}\right)\sideset{}{^\star}\sum_{\beta \bmod q^{2\ell-j}}\left(\frac{\beta }{q^r}\right)e\left(\frac{-\overline{c_1\beta}q^{r}\delta+\overline{2}\beta c_2}{q^{2\ell-j}}\right)=\left(\frac{\delta uvnm}{q^r}\right)S_r(\bar 2c_2,-\overline{c_1}q^r\delta;q^{2\ell-j}).
$$
The Kloosterman type sum $S_r(\bar 2c_2,-\overline{c_1}q^r\delta;q^{2\ell-j})$ is free of $n$, $m$, and it is bounded above by $4(c_2,q^r)q^{\ell-\frac{j}{2}}$. 
\begin{lemma}
\label{char31}
Let $u,v,n,m$ be as in \eqref{tr'}. Then
$$
E_{r,c_1}(c_2;u\delta n,v\delta m)=g_{\delta,u,v}^\star(c_2)g(n,c_2)g(m,c_2)\left(\frac{\delta uvnm}{q^r}\right)\left(\frac{\delta}{nm}\right)S_r(\bar 2c_2,-\overline{c_1}q^r\delta;q^{2\ell-j}).
$$
\end{lemma}
Since $r\leq 2\ell-j-2$ the character sum $S_r(\bar 2c_2,-\overline{c_1}q^r\delta;q^{2\ell-j})=0$ if $c_2=0$, or if $q^r\nmid c_2$.  
\begin{lemma}
We have
$$
E_{r,c_1}(0;u\delta n,v\delta m)=0.
$$
\end{lemma}
 
\vspace{.5cm}

Substituting the explicit value of the character sum in Lemma \ref{poisson2}, and interchanging the order of summations, we get 
\begin{align}
\label{trr}
T_{r,C}^\star(\delta,c_1,u,v)=&\frac{q^{\frac{3}{2}r}}{2q^{2\ell-j}\sqrt{C}}\frac{1}{c_1\delta uv}\left(\frac{\delta}{q^r}\right)\sum_{c_2\in\mathbb Z-\{0\}}g_{\delta,u,v}^\star(c_2)S_r(\bar 2c_2,-\overline{c_1}\delta;q^{2\ell-j-r})
\\
\nonumber &\times \mathop{\sum\sum}_{\substack{n,m=1\\(n,m)=1\\(nm,\delta q)=1\\ n,m\equiv 1 \bmod 4}}^{\infty}b_{\delta,\bar\chi}(n,c_2)b_{\delta,\chi}(m,c_2)\frac{\tilde{\mathcal I}_{r,c_1}(q^rc_2;\delta un,\delta vm)}{nm}
\end{align}
where the new coefficients are given by
$$
b_{\delta,\psi}(n,c_2)=\psi(n)\left(\frac{q^rc_{11}\delta}{n}\right)g(n,c_2).
$$

\vspace{.5cm}

Next we will analyse the integral 
\begin{align}
\label{last-int-again}
\tilde{\mathcal I}_{r,c_1}(c_2;n,m)=\int G(z)I_{r}\left(2n,2m;\frac{Cz}{q^r}\right)
e\left(\frac{q^{2r}nm}{q^{2\ell-j}Cz}-\frac{c_2Cz}{2q^{2\ell-j+r}c_1[n,m]}\right)\frac{dz}{z^{\frac{3}{2}}}
\end{align} 
which appears in \eqref{trr} and is defined in Lemma~\ref{poisson2}. We will obtain bounds for this integral, which in particular will also give us the effective range for the $c_2$ sum in \eqref{trr}. Also we need to separate the variables $n$ and $m$ to pave the way for an application of the large sieve. Recall that we are taking $n,m>0$. \\

Let us temporarily write $\delta n$ and $\delta m$ in place of $n$ and $m$ respectively, and put the restriction $(n,m)=1$. We replace the integral representation \eqref{int-again} to obtain 
\begin{align*}
\tilde{\mathcal I}_{r,c_1}&(c_2;\delta n,\delta m)=\iiint h\left(x\right)h\left(y\right)G(z) W_k\left(\frac{2xyN^2}{q^jCz}\right)\\
&\times e\left(\frac{4xyN^2+q^{2r+2j-2\ell}\delta^2nm-2\delta(nx+my)Nq^{r+j-\ell}}{q^jCz}-\frac{c_2Cz}{2q^{2\ell-j+r}c_1\delta nm}\right)\frac{dz}{z^{\frac{3}{2}}}dxdy.
\end{align*}
Using repeated integration by parts in the $z$-integral, and using the bounds for $n$ and $m$ from Section~\ref{sec-char}, it follows that the integral is negligibly small unless
$$
|c_2|\ll \frac{c_1q^{4\ell-2j+\ell\varepsilon}}{\delta q^r}\left(B^2+B^{-2}\right).
$$
This gives the effective range for the $c_2$ sum in \eqref{trr}. This is good enough for our purpose for $B<q^{\ell\varepsilon}$. However for larger $B$ we will obtain a better range below.\\

To get a partial separation of the variables $n$ and $m$, we define new variables $x'=x/m$, $y'=y/n$ and $z'=z/nm$. With this change of variables $\tilde{\mathcal I}_{r,c_1}(c_2;\delta n,\delta m)$ reduces to
\begin{align*}
\sqrt{nm}\iiint h\left(mx'\right)h\left(ny'\right)G(nmz')W_k\left(\frac{2x'y'N^2}{q^jCz'}\right)e\left(\frac{4B\Delta(x',y')}{z'}-\frac{c_2Cz'}{2q^{2\ell-j+r}c_1\delta}\right)\frac{dz'}{z'^{\frac{3}{2}}}dx'dy',
\end{align*}
where 
$$
\Delta(x',y')=\left(x'-\frac{\delta q^{r+j-\ell}}{2N}\right)\left(y'-\frac{\delta q^{r+j-\ell}}{2N}\right).
$$ 
Now suppose $B>q^{\ell\varepsilon}$, so that by Lemma \ref{int-000} we have $|\delta m|\asymp N/q^{r+j-\ell}$ (upto a factor of size $q^{\ell\varepsilon}$). Then for a given $x'$, by repeated integration by parts in the $y'$ integral we obtain that $\tilde{\mathcal I}_{r,c_1}(c_2;\delta n,\delta m)$ is negligibly small unless 
$$
\left|x'-\frac{\delta q^{r+j-\ell}}{2N}\right|\ll \frac{\delta q^{r+j-\ell+\ell\varepsilon}}{NB}. 
$$
So we get that upto a negligibly error the integral $\tilde{\mathcal I}_{r,c_1}(c_2;\delta n,\delta m)$ is given by
\begin{align*}
\sqrt{nm}\mathop{\iint}_{\left|x'-\frac{\delta q^{r+j-\ell}}{2N}\right|\ll \frac{\delta q^{r+j-\ell}}{NB}q^{\ell\varepsilon}} &h\left(mx'\right)h\left(ny'\right)\\
&\times \int G(nmz')W_k\left(\tfrac{2x'y'B}{z'}\right)e\left(\tfrac{4B\Delta(x',y')}{z'}-\tfrac{c_2Cz'}{2q^{2\ell-j+r}c_1\delta}\right)\frac{dz'}{z'^{\frac{3}{2}}}dx'dy'.
\end{align*} 
Now we can get a refined effective range for $c_2$ by integrating by parts the inner integral. Notice that the restriction on $x'$ implies that $B\Delta(x',y')/z' \ll q^{\ell\varepsilon}$. So we get that the inner integral is negligibly small unless 
$$
|c_2|\ll \frac{c_1q^{4\ell-2j+\ell\varepsilon}}{\delta q^r}B.
$$
Observe that compared to the previous bound we have saved an extra $B$.   \\

We summarize our findings in the following lemma. Here $\tilde G(s)$ and $\tilde h(s)$ denote the Mellin transform of the smooth compactly supported functions $G$ and $h$ respectively. 
\begin{lemma}
\label{all-about-int}
The integral $\tilde{\mathcal I}_{r,c_1}(q^rc_2;\delta n,\delta m)$ is negligibly small unless 
$$
|n|, |m| \ll \frac{N}{\delta q^{j+r-\ell}}\left(B^{-1}+1\right)q^{\ell\varepsilon},\;\;\;\text{and}\;\;\;|c_2|\ll \frac{c_1q^{4\ell-2j+\ell\varepsilon}}{\delta q^{2r}}(B+B^{-2}).
$$
Moreover upto a negligible error term we have
\begin{align}
\label{one-more}
\tilde{\mathcal I}_{r,c_1}(q^rc_2;\delta n,\delta m)= \frac{1}{(2\pi i)^3}\mathop{\iiint}_{(\sigma_1),(\sigma_2),(\sigma_3)}\tilde h(s_1)\tilde h(s_2)\tilde G(s_3)\mathcal F_{r,c_1}(c_2;\delta,B;\mathbf{s})\frac{d\mathbf{s}}{m^{s_1+s_3-\frac{1}{2}}n^{s_2+s_3-\frac{1}{2}}},
\end{align} 
where
\begin{align*}
\mathcal F_{r,c_1}(c_2;\delta,B;\mathbf{s})=\mathop{\iiint}_{\mathcal R(B)} x^{-s_1}y^{-s_2}z^{-s_3}W_k\left(\tfrac{2xyB}{z}\right)e\left(\tfrac{4B\Delta(x,y)}{z}-\tfrac{c_2Cz}{2q^{2\ell-j}c_1\delta}\right)\frac{dxdydz}{z^{\frac{3}{2}}}.
\end{align*} 
For $B<q^{\ell\varepsilon}$ we take the region $\mathcal R(B)=[Q^{-1},Q]^3\cap \{2^{-1}\leq xy/z \leq 4\}$, and for $B\geq q^{\ell\varepsilon}$ the region is obtained by putting the further restriction that 
$$
\left|x-\frac{\delta q^{r+j-\ell}}{2N}\right|\ll \frac{\delta q^{r+j-\ell}}{NB}q^{\ell\varepsilon}.
$$
\end{lemma}

In the integral \eqref{one-more} we will take the location of the contours to be $\sigma_1=\sigma_2=1$ and $\sigma_3=-\frac{1}{2}+\varepsilon$. For this choice we have
\begin{align}
\label{bound-for-F}
\mathcal F_{r,c_1}(c_2;\delta,B;\mathbf{s})\ll \min\{B^{-\frac{1}{2}},B^{k-1}\}\min\{1,B^{-1}\}q^{\ell\varepsilon}\ll \min\{B^{-\frac{3}{2}},B^{k-1}\}q^{\ell\varepsilon}.
\end{align}
This is obtained by trivially estimating the integrals over $x$, $y$ and $z$, taking into account the size of $W_k$ (see \eqref{bessel-bd0}) and the localization of $x$ for $B>q^{\ell\varepsilon}$. Observe that the integral over $s_1$, $s_2$ and $s_3$ converges absolutely due to the rapid decay of the Mellin transforms as $|\text{Im}(s_i)|\rightarrow\infty$.


\section{Separation of variable and large sieve}
\label{sec-large}

Using the integral representation Lemma \ref{all-about-int}, we will now analyse the contribution of the positive frequencies, i.e. $c_2>0$, to $T_{r,C}^\star(\delta,c_1,u,v)$ in \eqref{trr}. To this end we need to get bounds for 
\begin{align}
\label{sumsum}
\frac{q^{\frac{3}{2}r}\sqrt{B}}{Nq^{2\ell-\frac{3}{2}j}c_1\delta uv}&\mathop{\iiint}_{(\mathbf{\sigma})}\frac{\tilde h(s_1)\tilde h(s_2)\tilde G(s_3)}{u^{\frac{1}{2}+s_2+s_3}v^{\frac{1}{2}+s_1+s_3}}\mathcal B(\mathbf{s})d\mathbf{s},
\end{align}
where 
\begin{align}
\label{bsum}
\mathcal B(\mathbf{s})=\sum_{c_2=1}^{\infty}& g_{\delta,u,v}^\star(c_2) S_r(c_2)\mathcal F_{r,c_1}(c_2;\delta,B;\mathbf{s})\mathop{\sum\sum}_{\substack{n,m=1\\(n,m)=1\\(nm,2\delta q)=1}}^{\infty}\frac{b_{\delta,\bar\chi}(n,c_2)}{n^{\frac{1}{2}+s_2+s_3}}\frac{b_{\delta,\chi}(m,c_2)}{m^{\frac{1}{2}+s_1+s_3}}.
\end{align}
Here we are using the shorthand notation $S_r(c_2)=S_r(\bar 2c_2,-\overline{c_1}\delta;q^{2\ell-j-r})$, for which we will use the Weil bound $|S_r(c_2)|\ll q^{\ell-\frac{1}{2}(j+r)}$. To ensure absolute convergence in the inner sums, a priori we put the restrictions that $\sigma_1+\sigma_3>1$ and $\sigma_2+\sigma_3>1$.  Also notice that we have extended the sums over $n$ and $m$ to all odd integers, for this manoeuvre we need to introduce an extra character modulo $4$, which we are going to ignore. Also we need to replace the Gauss sum $g(n,c_2)$ (also $g(m,c_2)$), which appears in the coefficients $b_{\delta,\psi}(n,c_2)$, by the multiplicative function 
$$
G_{c_2}(n)=\left(\frac{1-i}{2}+\left(\frac{-1}{n}\right)\frac{1+i}{2}\right)g(n,c_2).
$$
To separate the sums over $n$ and $m$ in \eqref{bsum}, we use Mobius inversion to get 
\begin{align*}
\mathcal B(\mathbf{s})=\sum_{\substack{\theta=1\\(\theta,2\delta q)=1}}^{\infty}\mu(\theta)&\sum_{c_2=1}^\infty g_{\delta,u,v}^\star(c_2) S_r(c_2)\mathcal F_{r,c_1}(c_2;\delta,B;\mathbf{s})\mathop{\sum\sum}_{\substack{n,m=1\\(nm,2\delta q)=1}}^{\infty}\frac{b_{\delta,\bar\chi}(\theta n,c_2)}{(\theta n)^{\frac{1}{2}+s_2+s_3}}\frac{b_{\delta,\chi}(\theta m,c_2)}{(\theta m)^{\frac{1}{2}+s_1+s_3}}.\\
\end{align*}

Now we consider the $L$-series given by
$$
L_{\theta,\delta,\chi}(s;c_2)=\mathop{\sum}_{\substack{n=1\\(n,2\delta q)=1}}^{\infty}\frac{b_{\delta,\chi}(\theta n,c_2)}{(\theta n)^{\frac{1}{2}+s}}=\mathop{\sum}_{\substack{n=1\\(n,2\delta q)=1}}^{\infty}\frac{\chi(\theta n)G_{c_2}(\theta n)}{(\theta n)^{\frac{1}{2}+s}}\left(\frac{n\theta}{q^rc_{11}\delta}\right).
$$
The multiplicativity of the coefficients yield an Euler product for the series. We write $c_1=c_{11}c_{12}^2$, $c_2=c_{21}c_{22}^2$ and $\delta=\delta_1\delta_2^2$ with $c_{11}$, $c_{21}$ and $\delta_1$ square-free. Then it follows that we have a factorization
$$
L_{\theta,\delta,\chi}(s;c_2)=L\left(s,\chi\left(\frac{q^r\delta_1c_{11}c_{21}}{.}\right)\right)\tilde L_{\theta,\delta,\chi}(s;c_2),
$$
where the $L$-series $\tilde L_{\theta,\delta,\chi}(s;c_2)$ converges absolutely in the region $\sigma>\frac{1}{2}+\varepsilon$. Moreover in this domain we have
$$
\tilde L_{\theta,\delta,\chi}(s;c_2)\ll_{\varepsilon}\frac{(q^{\ell}\delta c_1c_2)^{\varepsilon}}{\theta^{\frac{1}{2}+\varepsilon}}.
$$\\

Using the $L$-series we can write
\begin{align*}
\mathcal B(\mathbf{s})=\sum_{\substack{\theta=1\\(\theta,2\delta q)=1}}^{\infty}\mu(\theta)&\sum_{c_2=1}^\infty g_{\delta,u,v}^\star(c_2) S_r(c_2)\mathcal F_{r,c_1}(c_2;\delta,B;\mathbf{s})L_{\theta,\delta,\bar\chi}(s_2+s_3;c_2)L_{\theta,\delta,\chi}(s_1+s_3;c_2).
\end{align*}
We move the contours to $\sigma_1=\sigma_2=1$ and $\sigma_3=-\frac{1}{2}+\varepsilon$. Then applying Cauchy on the sum over $c_2$,  we are led to consider 
\begin{align}
\label{tc}
\frac{q^{\frac{3}{2}r}\sqrt{B}}{Nq^{2\ell-\frac{3}{2}j}c_1\delta (uv)^2}&\mathop{\iiint}_{(\mathbf{\sigma})}|\tilde h(s_1)\tilde h(s_2)\tilde G(s_3)|\mathcal B^\star_X(\mathbf{s})|d\mathbf{s}|,
\end{align}
where $\sigma$ is as above, and 
\begin{align*}
\mathcal B^\star_{X}(\mathbf{s})=\sum_{\theta=1}^{\infty}\frac{1}{\theta^{1+\varepsilon}}\sum_{c_{22}\ll Q}\sum_{c_{21}\sim X/c_{22}^2}\left|g_{\delta,u,v}^\star(c_2) S_r(c_2)\mathcal F_{r,c_1}(c_2;\delta,B;\mathbf{s})\right|\left|L\left(s_1+s_3,\chi\left(\tfrac{q^r\delta_1c_{11}c_{21}}{.}\right)\right)\right|^2.
\end{align*}
Using approximate functional equation we can express the Dirichlet $L$-function as a rapidly converging series with effective length given by the square-root of the analytic conductor. The analytic conductor is given by $[q^\ell c_{11},\frac{c_{21}\delta_1}{(c_{21},\delta_1)^2}]\left(3+|t_1+t_3|\right)$. Using the main result of \cite{HB} we get
\begin{align*}
\sum_{c_{22}\ll Q}(\delta_1\delta_2,c_{22})\sum_{\substack{c_{21}\sim X/c_{22}^2\\q\nmid c_{21}}}(\delta_1,c_{21})\left|L\left(s_1+s_3,\chi\left(\tfrac{q^r\delta_1 c_{11}c_{21}}{.}\right)\right)\right|^2\ll (q^{\ell}T_{13})^{\varepsilon}T_{13}\left(X+\sqrt{q^\ell c_{11}\delta_1X}\right),
\end{align*}
where $T_{13}=\left(3+|t_1+t_3|\right)$.(Recall that $c_1=c_{11}c_{12}^2$, $c_2=c_{21}c_{22}^2$ and $\delta=\delta_1\delta_2^2$ with $c_{11}$, $c_{21}$ and $\delta_1$ square-free.) \\

Substituting this bound in \eqref{tc}, using the bounds from Lemma \ref{bound-for-g}, $|S_r(c_2)|\ll q^{\ell-\frac{1}{2}(j+r)}$ (which is the Weil bound for Kloosterman sums) and \eqref{bound-for-F}, we obtain 
\begin{align}
\label{tc2}
\mathcal B_X^\star(\mathbf{s})\ll uv\delta_2 q^{\ell-\frac{1}{2}(j+r)}\min\{B^{-\frac{3}{2}},B^{k-1}\}T_{13}\left(X+\sqrt{q^\ell c_{11}\delta_1 X}\right)(q^{\ell}T_{13})^{\varepsilon}.
\end{align}
According to Lemma \ref{all-about-int} in the worst case scenario $X=\frac{c_1q^{4\ell-2j+\ell\varepsilon}}{\delta q^{2r}}(B+B^{-2})$. Using \eqref{tc}, it follows that the contribution of the positive frequencies $c_2>0$ to \eqref{trr} is dominated by
$$
\frac{q^{j-\ell+r}}{Nc_1\delta_1\delta_2 uv}\min\{B^{-1},B^{k-\frac{1}{2}}\}\max\{B,B^{-2}\}\left(\frac{c_1q^{4\ell-2j}}{\delta q^{2r}}+\sqrt{\frac{c_{11}c_1q^{5\ell-2j}}{\delta_2^2q^{2r}}}\right)q^{\ell\varepsilon}.
$$
Now we observe that 
$$
\min\{B^{-1},B^{k-\frac{1}{2}}\}\max\{B,B^{-2}\}\ll 1
$$ 
(as $k>2$). Bounding the contribution of the negative frequencies $c_2<0$ in exactly the same manner we obtain
\begin{align}
\label{trr-final-bd}
T_{r,C}^\star(\delta,c_1,u,v)\ll \frac{q^{j-\ell+r}}{Nc_1\delta_1\delta_2 uv}\left(\frac{c_1q^{4\ell-2j}}{\delta q^{2r}}+\sqrt{\frac{c_{11}c_1q^{5\ell-2j}}{\delta_2^2q^{2r}}}\right)q^{\ell\varepsilon}.
\end{align}
Substituting in \eqref{tr'} it follows that
$$
\mathcal T_{r,C}^\star\ll Nq^{-\frac{1}{2}r+\ell\varepsilon}\sum_{\delta\leq Q}\frac{1}{\delta_1\delta_2}
\sum_{\substack{c_1\leq Q\\c_1|(2\delta)^{\infty}}}|\tilde C(\delta,c_1)|\left(\frac{q^{2\ell-\frac{3}{2}j}}{\delta_1\delta_2^2}+\frac{q^{\frac{1}{2}(\ell-j)}}{\delta_2c_{12}}\right).
$$
Now we use Lemma \ref{heck-trick} to execute the remaining sums on $\delta$ and $c_1$. Summing over all $r$, we see that the optimal choice of $j$ is given by $j=2[\theta \ell]$ with $\theta =2/3$.



\end{document}